\documentclass[12pt]{amsart}

\theoremstyle{definition}
\newtheorem{definition}{Definition}
\theoremstyle{plain}
\newtheorem{theorem}{Theorem}
\newtheorem{corollary}{Corollary}
\newtheorem{lem}{Lemma}
\newtheorem{prop}{Proposition}
\newtheorem{ejem}{Example}
\theoremstyle{remark}
\newtheorem{remark}{Remark}

\DeclareMathOperator{\CC}{\mathbb C}
\DeclareMathOperator{\RR}{\mathbb R}

\DeclareMathOperator{\ZZ}{\mathbb Z}
\DeclareMathOperator{\NN}{\mathbb N}

\DeclareMathOperator{\Lim}{Lim}
\DeclareMathOperator{\essinf}{ess\:inf}
\DeclareMathOperator{\essup}{ess\:sup}

\def\limsup{\mathop{\overline{\rm lim}}}

\title{Inverses of moment hermitian matrices.}
\author{C. Escribano} 
\address{C. Escribano: Departamento de Matem\'atica Aplicada, Facultad de Inform\'atica de Madrid\\
         Universidad Polit\'ec\-ni\-ca, Campus de Montegancedo\\
         Boadilla del Monte, 28660 , Madrid, Spain\\ Phone +34 913367419, Fax +34 913367426}
\email{cescribano@fi.upm.es}
\author{R. Gonzalo} 

\address{R. Gonzalo: Departamento de Matem\'atica Aplicada, Facultad de Inform\'atica de Madrid\\
         Universidad Polit\'ec\-ni\-ca, Campus de Montegancedo\\
         Boadilla del Monte, 28660, Madrid, Spain\\ Phone +34 913367419, Fax +34 913367426}
\email{rngonzalo@fi.upm.es}
\author{E. Torrano} 

\address{E. Torrano: Departamento de Matem\'atica Aplicada, Facultad de Inform\'atica de Madrid\\
         Universidad Polit\'ec\-ni\-ca, Campus de Montegancedo\\
         Boadilla del Monte, 28660, Madrid, Spain\\ Phone +34 913367419, Fax +34 913367426}

\email{emilio@fi.upm.es}

\begin{document}

\begin{abstract} Motivated by \cite{Berg-Szwarc} we study the existence of the inverse of infinite Hermitian moment matrices associated with
measures with support on the complex plane.  We relate this problem to  the asymptotic behaviour of the smallest eigenvalues of
 finite sections  and we study it from the point of view of  infinite transition matrices associated to the orthogonal polynomials.  For  Toeplitz matrices we introduce the notion of weakly asymptotic Toeplitz matrix and we show that, under certain assumptions, the inverse of
a Toeplitz moment matrix is weakly asymptotic Toeplitz. Such  inverses are computed in terms of some limits of the coefficients of the associated orthogonal polynomials. We finally show that the asymptotic behaviour of the smallest eigenvalue of a moment Toeplitz matrix only depends on the absolutely part of the associated measure.

\end{abstract}

\maketitle
\begin{quotation} {\sc {\footnotesize Keywords}}. {\small Hermitian moment problem,  orthogonal polynomials,
smallest eigenvalue, measures, inverses of infinite matrices}
\end{quotation}

\smallskip

\begin{quotation} {\it  MSC 2010 Subject Classification 44A60, 15A29}
\end{quotation}
\section{Introduction}

  Let $\mathbf{M}=(c_{i,j})_{i,j=0}^{\infty}$ be an infinite Hermitian matrix, i.e., $c_{i,j}= \overline{c_{j,i}}$ for all
$i,j$ non-negative integers. Following \cite{EGT} we say that $\mathbf{M}$ is positive definite (in short, an HPD matrix) if
$\vert \mathbf{M}_n \vert >0$ for all $n\geq 0$, where $\mathbf{M}_n$ is the truncated matrix of size $(n+1)\times (n+1)$ of $\mathbf{M}$. An
HPD matrix defines an inner product $\langle \;,\; \rangle$ in the space $\mathbb{P}[z]$ of all polynomials with complex coefficientes in the following way:
if  $p(z)=\sum_{k=0}^{n}v_kz^k$ y
$q(z)=\sum_{k=0}^{m}w_kz^k$ then
\begin{equation} \label{inner}
\langle p(z),q(z)\rangle=v\mathbf{M}w^{*},
\end{equation}
 being $v=(v_0,\dots,v_n,0,0, \dots), w=(w_0,\dots,w_m,0,0, \dots) \in c_{00}$ where $c_{00}$ is the space
of all complex sequences with only finitely many non-zero entries. In the case of an HPD matrix $\mathbf{M}$ being a moment matrix
associated with a measure $\mu$, i.e., whenever there exists a measure $\mu$ with support on $\CC$ such that for all $i,j\geq 0$
$$
c_{i,j}= \int z^{i} \overline{z}^jd\mu(z),
$$
 we denote $\mathbf{M}=\mathbf{M}(\mu)$. Note that in this case
the inner product induced by $\mathbf{M}(\mu)$ is the inner product in $L^{2}(\mu)$:
$$
\langle p(z),q(z)\rangle =v\mathbf{M}(\mu)w^{*}= \int p(z)\overline{q(z)}d\mu.   \label{inner}
$$

 For more information concerning the characterization of HPD matrices which are
moment matrices with respect to a certain measure $\mu$ with support on $\CC$ see among others  \cite{Atzmon},\cite{Berg-Maserick} and  \cite{Szafraniec}.

 An HPD matrix $\mathbf{M}$ is the Gram matrix of the inner product \ref{inner} in the vector space $\mathbb{P}[z]$ with respect
to $\{z^n\}_{n=0}^{\infty}$, i.e., $\mathbf{M}=(\langle z^i,z^j\rangle)_{i,j=0}^{\infty}$. Let $\{P_n(z)\}_{n=0}^{\infty}$ denote the sequence of orthonormal polynomials induced by such inner product,
uniquely determined by the requirements that
$$
P_n(z) = \sum_{k=0}^{n}
b_{k,n}z^k,
$$
 with positive leading coefficient $b_{n,n}>0$ and satisfying the orthonormality condition. In the case of $\mathbf{M}(\mu)$ being a moment matrix associated with a certain measure $\mu$ then $\{P_n(z)\}_{n=0}^{\infty}$ is the sequence of orthonormal polynomials with respect to $\mu$.

 We denote by $\lambda_n$ the smallest eigenvalue of $\mathbf{M}_n$. It is easy to check that the sequence
$\{\lambda_n\}_{n=0}^{\infty}$ is a non increasing positive sequence and therefore $\lim_{n\to \infty} \lambda_n$ exists.
In the case of a moment matrix $\mathbf{M}(\mu)$ we denote
$\lambda_n =\lambda_n(\mu)$.

 For positive definite
Hankel matrices, which are moment matrices associated with positive measures on $\RR$,  Berg, Chen and Ismail \cite{Berg} have proved that a measure $\mu$ on $\RR$ is {\it determinate}, meaning that $\mu$ is the only measure with real support having the same
moments as $\mu$, if and only if $\lim_{n\to \infty} \lambda_n(\mu) =0$. In this direction, in the case of  HPD moment matrices $\mathbf{M}(\mu)$ it is proved in \cite{EGT}
that $\lim_{n\to \infty} \lambda_n(\mu) =0$ is a necessary condition to assure the density of polynomials in the space $L^{2}(\mu)$ when compactly supported measures are considered, although it is not a sufficient condition.

 For an HPD matrix $\mathbf{M}$ let $\mathbf{B}=(b_{k,n})_{k,n=0}^{\infty}$ denote the infinite upper triangular matrix, i.e., $b_{k,n}=0$ whenever $k>n$, given by the coefficients of the orthonormal polynomials with respect to the inner product \ref{inner} induced by $\mathbf{M}
$. In \cite{Berg-Szwarc} the authors state that it is likely that, for Hankel matrices $\mathbf{H}$ in the indeterminate case
(i.e. whenever $\lim_{n\to \infty} \lambda_n >0$ as a consequence of \cite{Berg}), the identity
$\mathbf{H}_n \mathbf{A}_n= \mathbf{I}_n$ where $\mathbf{A}_n=\mathbf{B}_n\mathbf{B}_n^{*}$ extends to the infinite case in the sense that if $\mathbf{A}=\mathbf{B}\mathbf{B}^{*}$,
$$
\mathbf{H}\mathbf{A}=  \mathbf{A} \mathbf{H}=\mathbf{I}.
$$
 Motivated by this problem and taking in mind that for $\mathbf{M}$ being an HPD matrix it is verified  that
$\mathbf{M}_{n}^{-1}=\overline{\mathbf{B}_n}\mathbf{B}_n^{t}$ as it is proved in
\cite{Berg-Szwarc}, \cite{EGT} we here study the following problem:

\noindent
 {\bf Problem 1.} Let $\mathbf{M}(\mu)$ be an HPD moment matrix associated with a measure $\mu$ with support on $\CC$. Is the matrix $\mathbf{A}(\mu)= \overline{\mathbf{B}}\mathbf{B}^{t} $, whenever such formal matrix product exists, a classical inverse of the matrix $\mathbf{M}(\mu)$ in the sense that
$$
\mathbf{A}(\mu) \mathbf{M}(\mu) =\mathbf{M}(\mu) \mathbf{A}(\mu) =\mathbf{I}?
$$
 Note that in \cite{Sivakumar} it is showed that an infinite matrix could have several classical inverses; we here are
interested in the inversion of moment matrices in terms of the coefficients of the orthogonal polynomials given by the infinite matrix $\mathbf{B}$. We point out that the approach of determining the inverse of a finite Hankel or Toeplitz matrix in connection with the theory of orthogonal polynomials appears in \cite{Trench-I} and
\cite{Trench-II} where an algorithm for the inversion of such matrices is obtained. Later on, this point of view is also treated in
\cite{Provost} for certain finite moment Hankel matrices; in the finite dimensional case this kind of algorithms enables to compute inverse matrices faster.

 The paper is organized as follows: in Setion I we show that the answer of  Problem $1$ is negative in general. Moreover, we analyze the existence of  $\mathbf{A}(\mu)$ whenever $\lim_{n\to \infty} \lambda_n(\mu)>0$. In this direction we show that $\lim_{n\to \infty} \lambda_n(\mu)>0$ is a sufficient condition to assure the existence of $\mathbf{A}(\mu)$ but not to assure
that $\mathbf{A}(\mu)$ is a classical inverse of $\mathbf{M}(\mu)$. We also provide several examples in order to show that
the condition $\lim_{n\to \infty} \lambda_n(\mu)>0$ is not necessary to assure that $\mathbf{A}(\mu)$ exists and
is a classical inverse matrix of $\mathbf{M}(\mu)$.

 Section II is devoted to study the problem of the existence of the classical inverse of an HPD Toeplitz matrix
in terms of the coefficients of the orthnormal polynomials. It is well known that
the inverse of a Toeplitz matrix is not, in general, a Toeplitz matrix, even in the finite dimensional case. Nevertheless, the inverse of a finite
Toeplitz matrix is {\it persymmetric} as it is proved in \cite{zohar}. Recall
that a matrix $\mathbf{A}=(a_{i,j})_{i,j=0}^{n}$ is persymmetric  (see \cite{zohar}) when for every $0 \leq i,j \leq n$
$$
a_{i,j}= a_{n-j,n-i}.
$$
 In other words, a matrix is persymmetric if it has symmetry about its cross diagonal (the diagonal extended from the
upper right corner to the lower left corner). We show that this property has a great impact in the form of a classical
inverse of certain infinite Toeplitz matrices. In order to do it we introduce the notion of {\it weakly asymptotic Toeplitz matrix} very closely
related to the notion of weakly asymptotic Toeplitz operator that appears in \cite{Barria-Halmos}; indeed, in the particular case of matrices defining bounded operators in Hilbert spaces, these matrices are the representations of such operators with respect to orthonormal
basis. In this direction we show that under certain assumptions the classical inverse of an HPD Toeplitz matrix is a weakly asymptotic Toeplitz matrix. Moreover,
we give the description of a classical inverse of a Toeplitz matrix in terms of the limits of their diagonals.
In all of the examples of HPD moment Toeplitz matrices considered we compute the associated matrix  $\mathbf{B}$.

 It is well known (see e.g. \cite{aki}) that every HPD Toeplitz matrix $\mathbf{T}$ is indeed the moment matrix $\mathbf{T}(\nu)$ for a certain measure $\nu$ with support on $\mathbb{T}$. Applying our general results to such matrices we obtain some interesting
consequences for measures $\nu$ supported  on the unit circle $\mathbb{T}$. In particular, we show that whenever  $\lim_{n\to \infty} \lambda_n(\nu)>0$  we may assure the existence of all the limits $\lim_{n\to \infty} b_{n-k,n}(\nu)$ for every $k\geq 0$ being
  $P_n(z)=\sum_{k=0}^{n} b_{k,n}(\nu)z^k$ the orthogonal polynomials associated with $\nu$. Note that in the case
of the main coefficients the existence of  $\lim_{n\to \infty} b_{n,n}(\nu)$ was already known by
Szeg\"{o} theory.

 Finally, we consider HPD Toeplitz matrices from the point of view of
 representations of bounded operators on the Hardy-Hilbert space $\mathbf{H}^{2}$. The problem of the inverse of bounded Toeplitz operator $\mathbf{T}_{\varphi}$ associated with a symbol $\varphi$ has been widely studied (see e.g. \cite{Bottcher}). We apply our techniques to the inversion of  Toeplitz operators not necessarily bounded. As a consequence we obtain that the asymptotic limit of the inverse of a Toeplitz matrix $\mathbf{T}_{\varphi}$
associated with a continuous symbol verifying $\inf_{z\in \mathbb{T}} \varphi(z) >0$ is the Toeplitz matrix $\mathbf{T}_{\frac{1}{\varphi}}$.

 Last section is devoted to the study of the asymptotic behaviour of the smallest eigenvalues
of the absolutely part of a measure with support on $\mathbb{T}$ and its consequences in the inversion of moment Toeplitz matrices. In \cite{EGT} it is proved that the $n$ large asymptotic of the smallest eigenvalues of a measure $\mu$ with support on the closed unit disk $\overline{\mathbb{D}}$ has a harmonic behaviour in the sense that $\lim_{n\to \infty} \lambda_n(\mu)=\lim_{n\to \infty} \lambda_{n}(\nu)$ being $\nu=\mu/\mathbb{T}$. In this direction,
   we prove that such asymptotic behaviour only depends on the absolutely continuous part of  $\nu$ with respect to the Lebesgue measure ${\bf m}$. We apply this result to the problem of inversion of a Toeplitz  moment matrix $\mathbf{T}(\nu)$. Motivated by the fact that $\mathbf{A}(\nu)=\mathbf{A}(\nu_a)$ in Example $3$ we state the problem of the equality of such matrices in the general case. In this direction we obtain that whenever $\lim_{n\to \infty} \lambda_n(\mu)>0$ then the bounded operators defined by both matrices have
the same norm. Moreover,
assuming that $\lim_{n\to \infty} b_{n-k,n}(\nu)= \lim_{n\to \infty} b_{n-k,n}(\nu_a)$ for every $k \geq 0$ we are be able to prove that
$\mathbf{A}(\nu)=\mathbf{A}(\nu_a)$
 in the general case.

 First, we introduce some notation. For $\mathbf{A}=(a_{i,j})_{i,j=0}^{n}$ being a finite matrix we identify the linear operator on $\CC^{n+1}$ induced by $\mathbf{A}$ with its matrix with respect to the standard basis of $\CC^{n+1}$.
Nevertheless in the infinite case we distinguish infinite matrices and operators using different notation: an infinite matrix $\mathbf{A}=(a_{i,j})_{i,j=0}^{\infty}$ defines a linear operator from the sequence space $\ell_2$ to $\ell_2$ if for every $x=(x_i)_{i=0}^{\infty}\in \ell_2$ the formal matrix product $\mathbf{A}x^{t}=(\sum_{j=1}^{\infty} a_{i,j}x_i )_{i=0}^{\infty}$ is defined and
belongs to $\ell_2$. If this operator is bounded it will be denoted by $\mathcal{A}$. Note that there are infinite matrices not defining bounded operators, consider for example the diagonal matrix $\mathbf{A}=(i\delta_{i,j})_{i,j=0}^{\infty}$. In \cite{Crone} a criterion has been proved to characterize when an infinite matrix defines a bounded operator from
$\ell_2$ to $\ell_2$. On the other hand, if $\mathcal{A}:\mathcal{H} \to \mathcal{H} $ is a bounded operator on a Hilbert space
$\mathcal{H}$ the representation of $\mathcal{A}$ using
an orthonormal basis $\mathfrak{B}= \{v_{i}\}_{i=0}^{\infty}$ in $\mathcal{H}
$ is by constructing the infinite matrix with entries
$a_{i,j}=\langle \mathcal{A}v_j,v_i \rangle$, i.e., $\mathbf{A}=(a_{i,j})_{i,j=0}^{\infty}$. On the Hilbert space $\ell_2$ we always use the
standard basis $\{e_n\}_{n=0}^{\infty}$.

\section{Infinite Transition matrices. Inversion of certain Hermitian moment
matrices. }

 Let
$\mathbf{M}=(c_{i,j})_{i,j=0}^{\infty} $ be an HPD matrix, in the
space $\mathbb{P}[z]$ two algebraical basis can be considered: the
basis which consists of the monomials $\mathfrak{B}
=\{z^{n}\}_{n=0}^{\infty}$ and the one that consists of the orthonormal polynomials
$\mathfrak{B}'=\{P_n(z)\}_{n=0}^{\infty}$ associated with the matrix $\mathbf{M}$ . We consider the coordinates of each element of the basis $\mathfrak{B}'$ with respect to $\mathfrak{B}$,
thus for every $n\in \NN_{0}$, $
P_n(z) = \sum_{k=0}^{n}
b_{k,n}z^k,
$ and let define the infinite upper triangular matrix $\mathbf{B}=(b_{k,n})_{k,n=0}^{\infty}$ with $b_{k,n}=0$
if $k>n$.

 Let  $n\in \NN_0$ be fixed, as usual $\mathbb{P}_n[z]$ denotes the space of polynomials of degree less or equal than $n$. The finite dimensional matrix  $\mathbf{B}_n$ is the transition matrix from the basis $\mathfrak{B}_n'=\{P_0(z), P_1(z),
\dots,P_n(z)\}$ in $\mathbb{P}_n[z]$ to
the basis $\mathfrak{B}_n=\{1, z,
\dots,z^n\}$. Since $\mathbf{M}_n,\mathbf{I}_n$ are both matricial representations of the same inner product with respect to $\mathfrak{B}_n,\mathfrak{B}_n'$ respectively then
$$
\mathbf{B}_n^{t}\mathbf{M}_n \overline{\mathbf{B}}_n=\mathbf{I}_n,
$$
 and consequently $\mathbf{M}_{n}^{-1} =\overline{\mathbf{B}_n}\mathbf{B}_n^{t}$ as it is proved in
\cite{Berg-Szwarc}, \cite{EGT} using kernel functions.

 As in the finite dimensional case the matrix $\mathbf{B}$ can be considered as the transition matrix from the basis $\mathfrak{B}'$ to $\mathfrak{B}$ in the following sense, if $p(z)=\sum_{k=0}^{n} v_kz^k= \sum_{k=0}^{n} w_kP_k(x)$ and
$(v_0,\dots, v_n,0,0,\dots),(w_0,\dots, w_n,0,0,\dots) \in c_{00}$ and $n\in \NN$:
$$
\left(
                    \begin{array}{c}
                      v_0 \\
                      \vdots \\
                     v_n \\
                      0  \\
                      \vdots \\
                    \end{array}
                  \right)
=
\mathbf{B}\left(
                    \begin{array}{c}
                      w_0 \\
                      \vdots \\
                      w_n \\
                      0  \\
                      \vdots \\
                    \end{array}
                  \right).
$$
 From now on we call $\mathbf{B}$
the {\it  transition matrix }  associated with the HPD matrix $\mathbf{M}$. In the case of
$\mathbf{M}$ being a moment matrix $\mathbf{M}(\mu)$ let denote $\mathbf{B}=\mathbf{B}(\mu)$.
\begin{remark} Note that if
the Cholesky decomposition of the Hermitian matrix $\mathbf{M}$ is given by $\mathbf{M}=\mathbf{L}\mathbf{L}^{*}$ then $\mathbf{B}^{t}=\mathbf{L}^{-1}$.
\end{remark}
\begin{definition} Let $\mathbf{M}$ be an HPD matrix and let $\mathbf{B}$ be the
transition matrix associated with $\mathbf{M}$. Let $\mathbf{A}=(a_{i,j})_{i,j=0}^{\infty}
$ be the  matrix
$\mathbf{A} = \overline{\mathbf{B}}\mathbf{B}^{t}$ whenever such formal matrix product is well defined, i.e. if for all $i,j \in \NN_0
$ there exists
$
a_{i,j}=\sum_{k=max\{i,j\}}^{\infty} \overline{b_{i,k}}b_{j,k}.
$
 In the case of $\mathbf{M}=\mathbf{M}(\mu)$ being a moment matrix associated with a measure $\mu$ we denote by
$\mathbf{A}=\mathbf{A}(\mu)$.
\end{definition}
\begin{remark} Note that the existence of the matrix product $\mathbf{A} = \overline{\mathbf{B}}\mathbf{B}^{t}$ is equivalent to the
fact that all the rows of the matrix $\mathbf{B}$ belong to $\ell_2$. 
\end{remark}

 The following result is essentially contained in \cite{Berg-Szwarc} in the case of Hankel matrices. The same result is true when Hermitian matrices are considered, not necessarily moment matrices. We include it
 for the sake of completeness :

\begin{lem} \label{lem1}
Let $\mathbf{M}$ be an infinite HPD matrix and let $\mathbf{B}$ be the  transition matrix associated with $\mathbf{\mathbf{M}}$. Then, the following are equivalent:

\begin{enumerate}
\item $\lim_{n\to \infty} \lambda_n = \lambda >0$.

\item The matrices $\mathbf{B}$ and $\mathbf{B}^{*}$ define bounded operators $\mathcal{B}$ and $\mathcal{B}^{*}$ on $\ell_2$ verifying
that $\Vert \mathcal{B} \Vert = \Vert \mathcal{B}^{*} \Vert =\lambda^{-1/2}$.
\end{enumerate}
\end{lem}
\begin{proof}  First of all we show that $
\Vert \mathbf{B}_n \Vert^2=\lambda_n^{-1}$
for every $n\in \NN_0$. Indeed,
$$
\lambda_n =
\min \{\dfrac{v\mathbf{M}_nv^{*}}{vv^{*}}: 0 \neq v\in \CC^{n+1}\} =
\frac{1}{\max  \{\dfrac{vv^{
*}}{v\mathbf{M}_nv^{*}} :
  0 \neq v\in
\CC^{n+1}\}}
$$
 and therefore:
\begin{eqnarray*}
\lambda_n^{-1}
&=& \max  \{\dfrac{vv^{*}}{v\mathbf{M}_nv^{*}} :  v\in
\CC^{n+1}\}= \max \{ vv^{*}:  v\in \CC^{n+1}, v\mathbf{M}_nv^{*}=1\}\\
\mbox{} & = &
\max  \{ w\mathbf{B}_n^{t}\overline{\mathbf{B}}_nw^{*} : w\in \CC^{n+1},
w\mathbf{B}_{n}^{t}\mathbf{M}_n\overline{\mathbf{B}}_nw^{*}=1\} \\
\mbox{} & = &
 \max  \{ \Vert \overline{\mathbf{B}}_nw^{*} \Vert^2 : w\in
\CC^{n+1}, ww^{*}=1\} \\
\mbox{} & = & \Vert \overline{\mathbf{B}}_n \Vert^2= \Vert \mathbf{B}_n \Vert^2.
\end{eqnarray*}

 Then, for every $n\in \NN_0$ we have $\Vert \mathbf{B}_n  \Vert = \Vert \mathbf{B}_n^{*} \Vert=\lambda_n^{-1/2}$. Consequently,
the sequence $\{\Vert \mathbf{B}_n \Vert\}_{n=0}^{\infty}$ is
bounded if and only if $\lim_{n\to \infty} \lambda_n >0$. Now, it is well known (see e.g. \cite{Crone}) that
$\mathbf{B}$ defines a bounded operator on $\ell_2$ if and only if $\sup_{n} \Vert \Pi_n \mathbf{B}^{*}\mathbf{B} \Pi_n \Vert <\infty$ where
$\Pi_n(x)$ is the $n-th$ section of the vector $x\in \ell_2$, i.e., $(\Pi_n(x))_i=x_i$ if $i\leq n$ and $(\Pi_n(x))_i=0$ if $i >n$.
Since $\Vert \Pi_n \mathbf{B}^{*}\mathbf{B} \Pi_n \Vert= \Vert \mathbf{B}_n \Vert$ it follows that  $\mathbf{B}$ defines a bounded operator $\mathcal{B}$
on $\ell_2$ if and only if
$\lim_{n \to \infty} \lambda_n =\lambda >0$ ; moreover, $\Vert \mathcal{B} \Vert =\lambda^{-1/2}$.

 Finally, if $\mathcal{B}$ is a bounded operator on $\ell_2$ then the adjoint $\mathcal{B}^{*}$ is
bounded and verifies $\Vert \mathcal{B}^{*} \Vert = \Vert\mathcal{B} \Vert =\lambda^{-1/2}$.

\end{proof}

 As a consequence of Lemma \ref{lem1} we give a sufficient condition in terms of the asymptotic  behaviour of $\lambda_n$
to assure the existence of the matrix $\mathbf{A}$ defining a bounded operator $\mathcal{A}$ on $\ell_2$.

\begin{lem} \label{lem2}
Let $\mathbf{M}$ be an HPD matrix such that $\lim_{n\to \infty} \lambda_n=\lambda >0$, then the
matrix $\mathbf{A}= \overline{\mathbf{B}}\mathbf{B}^t=(a_{i,j})_{i,j=0}^{\infty}$ exists
and defines the bounded operator $\mathcal{A}=\overline{\mathcal{B}}\mathcal{B}^{*}$ on $\ell_2$. Moreover,
\begin{enumerate}
\item $
a_{i,j}= \lim_{n\to \infty} \mathbf{M}_n^{-1}[i,j]$ for every $i,j \in \NN_0$.
\item $\Vert \mathcal{A} \Vert = \Vert \mathcal{B} \Vert^2=\lambda^{-1}$.
\end{enumerate}
\end{lem}

\begin{proof}
By Lemma \ref{lem1}
 the matrices $\overline{\mathbf{B}},\overline{\mathbf{B}}^{*}$ define bounded operators
on $\ell_2$ and consequently the rows and columns of such matrices belong to $\ell_2$. Then for every
$i,j\in \NN_0$ the series $\sum_{k=max\{i,j\}}^{\infty} \overline{b_{i,k}}b_{j,k}
$
is absolutely convergent and
$$
a_{i,j}=\sum_{k=max\{i,j\}}^{\infty}\overline{b_{i,k}}b_{j,k} = \lim_{n\to \infty} \sum_{k=max\{i,j\}}^n \overline{b_{i,k}}b_{j,k} = \lim_{n\to \infty}
\mathbf{M}_n^{-1}[i,j].
$$
 Therefore the matrix $\mathbf{A}$ exists and is the representation with respect to the standard basis of $\ell_2$
of the operator $\overline{\mathcal{B}}\mathcal{B}^{*}$. Moreover,$\Vert \mathcal{A} \Vert = \Vert \overline{\mathcal{B}}\mathcal{B}^* \Vert^2=\Vert \mathcal{B} \Vert^2=\lambda^{-1}$.
\end{proof}

 In general, when $\lim_{n \to \infty} \lambda_n(\mu)=0$ for a certain measure $\mu$ one can not even assure the existence of the
infinite matrix $\mathbf{A}(\mu)$ as the following example shows:

\begin{ejem} There exists an HPD moment matrix $\mathbf{M}(\mu)$ associated with a measure $\mu$ on $\CC$ such that $\lim_{n\to \infty} \lambda_n(\mu)=0$
and there does not exist the matrix $\mathbf{A}(\mu)$. Indeed, consider the Lebesgue measure $\mu$ in the circle
with center $(1,0)$ and radio $1$, and the associated moment matrix
$$
\mathbf{M}(\mu)= \left(%
\begin{array}{ccccc}
  1 & 1 & 1 & 1 & \hdots \\
  1 & 2 & 3 & 4 & \hdots \\
  1 & 3 & 6 & 10 & \hdots  \\
  1 & 4 & 10 & 20 & \hdots  \\
  \vdots & \vdots & \vdots & \vdots & \ddots \\
\end{array}%
\right).
$$
 The sequence of orthonormal polynomials associated with $\mu$  is $P_n(z)=(z-1)^n$ for all $n\geq 0$ and since $\sum_{k=0}^{\infty} \vert P_k(0) \vert^2 = \infty$ then the matrix
$\mathbf{A}(\mu)$ does not exist. Note that the transition matrix is $\mathbf{B}(\mu)=(b_{k,n})_{k,n=0}^{\infty}$ with
$b_{k,n}=(-1)^{n-k}\binom{n}{k}$ if $k \leq n$ and $b_{k,n}=0$ if $k>n$ which obviously does not define a bounded operator on $\ell_2$.
\end{ejem}

\begin{ejem} There exist  HPD moment matrices $\mathbf{M}(\mu)$ associated with
certain measures $\mu$ on $\CC$ such that $\lim_{n\to \infty} \lambda_n(\mu)=0$, the matrix $\mathbf{A}(\mu)$ exists and is a classical inverse of $\mathbf{M}(\mu)$. The easiest example is a diagonal matrix: indeed, let $\mu$ be the Lebesgue measure (uniform measure)
on the unit disk $\overline{\mathbb{D}}$; it is well known that $\mathbf{M}(\mu)= \left(c_{i,i}\delta_{i,j}\right)_{i,j=0}^{\infty}$ verifying
$c_{i,i}=\frac{\pi}{i+1}$. Then, $\mathbf{A}(\mu)$ is the diagonal matrix with entries $a_{i,i}=\frac{i+1}{\pi}$ verifying
$\mathbf{A}(\mu)\mathbf{M}(\mu)= \mathbf{M}(\mu)\mathbf{A}(\mu)=\mathbf{I}
$.

 We now provide another example: let $0<a<1$ and $\mathbf{M}=\left(a^{\max\{i,j\}}\right)_{i,j=0}^{\infty}$

$$\mathbf{M} =  \left (
\begin{array}{ccccc}
1 & a & a^2 & a^3 & \ldots \\
 a & a & a^2 & a^3 & \ldots \\
 a^2 & a^2 & a^2 & a^3 & \ldots \\
  a^3 & a^3 & a^3 & a^3
   & \ldots \\
  \vdots & \vdots & \vdots & \vdots & \ddots
\end{array}
\right).
$$
 It can be easily checked that $\mathbf{M}$ is positive definite since $\vert M_n \vert=a^{\frac{n(n+1)}{2}}(1-a)^n>0$ for
each $n\in \NN_0$. Consider the diagonal matrix $\mathbf{D}
=\left( a^{i/2} \delta_{i,j} \right)_{i,j=0}^{\infty}$ and the Toeplitz matrix
$\mathbf{T}=\left( a^{\frac{\vert i-j \vert}{2}} \right)_{i,j=0}^{\infty}$ then it holds that
$$
\mathbf{M}=\mathbf{D}^{t}\mathbf{T}\mathbf{D}.
$$
 Taking in account this equality it is obvious that $\mathbf{T}$ is an HPD Toeplitz matrix and consequently $\mathbf{T}=\mathbf{T}(\nu)$ for
a certain measure $\nu$ with support on $\mathbb{T}$. Using \cite{EST} and
\cite{Jorgensen} it follows
that  $\mathbf{M}$ is the moment matrix
of the image measure $\mu=\nu \circ f^{-1}$ under the transformation $f(z)=\sqrt{a}z$. In this particular case it can easily be obtained the orthonormal polynomials by the equations (see e.g. \cite{Szego}):
$$
P_n(z)=\dfrac{1}{\sqrt{\vert M_{n-1} M_{n}\vert}}
 \left|\begin{array}{cccc}
         1 & a & \hdots  & a^n \\
         a & a & \hdots & a^n \\
         \vdots & \vdots & \mbox{}   & \vdots \\
         1 & z &  \hdots& z^n \\
       \end{array}
   \right| = \frac{z^{n-1}}{\sqrt{a^n(1-a)}}(z-a) \,
     , n=
     1,2, \dots $$
 Consequently,
$$
\mathbf{B}(\mu)=\left(
             \begin{array}{ccccc}
               1 & \frac{-a}{\sqrt{a(1-a)}} & 0 & 0 & \hdots \\
               0 & \frac{1}{\sqrt{a(1-a)}} & \frac{-a}{\sqrt{a^2(1-a)}} & 0 & \hdots \\
               0 & 0 & \frac{1}{\sqrt{a^2(1-a)}} & \frac{-a}{\sqrt{a^3(1-a)}} & \hdots \\
               0 & 0 & 0 & \frac{1}{\sqrt{a^3(1-a)}} & \hdots \\
               \vdots & \vdots  & \vdots & \vdots & \ddots \\
             \end{array}
           \right).
$$
 In particular, $\mathbf{B}(\mu)
$ does not define a bounded operator on $\ell_2$ and by Lemma $1$ we
have that $\lim_{n\to \infty} \lambda_n(\mu)=0$. It can be checked $\mathbf{A}(\mu)\mathbf{M}(\mu)= \mathbf{M}(\mu)\mathbf{A}(\mu)=I$ where
$$
\mathbf{A}(\mu)
=\frac{1}{1-a}
\left(
             \begin{array}{ccccc}
               1 & -
               1 & 0 & 0 & \hdots \\
               -1 & \frac{a+1}{a} & \frac{-1}{a} & 0 & \hdots \\
               0 & \frac{-1}{a} & \frac{a+1}{a^2} & \frac{-1}{a^2} & \hdots \\
               0 & 0 & \frac{-1}{a^2} & \frac{a+1}{a^3} & \hdots
                \\
               \vdots & \vdots  & \vdots & \vdots & \ddots \\
             \end{array}
           \right).
$$
\end{ejem}

 The following example shows that the answer to Problem 1 is negative in the Hermitian case, even for Toeplitz  moment matrices:

\begin{ejem} There exists a moment Toeplitz matrix $\mathbf{T}(\nu)$ associated with a measure $\nu$ with support on $\mathbb{T}$ verifying $\lim_{n\to \infty} \lambda_n(\nu)>0$ and nevertheless
$$\mathbf{T}(\nu)\mathbf{A}(\nu) \neq \mathbf{I}, \qquad \mathbf{A}(\nu)\mathbf{T}(\nu) \neq \mathbf{I}.
$$
  Indeed, consider $\nu= \frac{1}{2}{\bf m} + \mu_{0}$ where
$\mu_{0}$ is the measure with support $\{0\}$ being a point mass with $\mu_{0}(\{0\})=
\frac{1}{2}$. Consider
$$
\mathbf{T}(\nu)= \begin{pmatrix}
  1  & \frac{1}{2}& \frac{1}{2} & \frac{1}{2} & \hdots  \\
  \frac{1}{2} & 1 & \frac{1}{2}  & \frac{1}{2} & \hdots  \\
 \frac{1}{2} & \frac{1}{2} & 1  & \frac{1}{2} & \hdots  \\
 \frac{1}{2} & \frac{1}{2} &  \frac{1}{2}   & 1 & \hdots   \\
  \vdots  & \vdots &
  \vdots  & \vdots  & \ddots
\end{pmatrix}.
$$
Since $\mathbf{T}(\nu) \geq \frac{1}{2}\mathbf{I}$ in the sense that for every $v\in c_{00}$ we have
$v\mathbf{T}(\nu)v^{*} \geq \frac{1}{2}vv^{*}$, it follows that $\lambda_n(\nu) \geq \frac{1}{2}$ for all $n\in \NN_0$ and $\lim_{n\to \infty} \lambda_n(\nu) \geq \frac{1}{2}$. It can be checked that
$\mathbf{M}_n^{-1}[i,j]= -\frac{2}{n+1}$ if $i\neq j$ and $\mathbf{M}_n^{-1}[i,i]= \frac{2n}{n+1}$ for each $n\in \NN_0$ and $0\leq i,j \leq n$. By using Lemma $2$
it follows that $a_{i,j}=\lim_{n\to \infty} \mathbf{M}_n^{-1}[i,j]= 2\delta_{i,j}$ and $\mathbf{A}(\nu)=2\mathbf{I}$. Therefore
$$\mathbf{A}(\mu)\mathbf{T}(\nu) \neq \mathbf{I} \qquad \qquad  \, \text{and } \qquad \qquad \mathbf{T}(\nu)\mathbf{A}(\nu)
 \neq \mathbf{I}.$$
Using the efficient numerical algorithms of Cholesky decomposition and inversion of lower triangular matrices we can determinate
the explicit representation of the
transition matrix $\mathbf{B}(\nu)=(b_{k,n})_{k,n=0}^{\infty}$ with $b_{n,n}=\frac{\sqrt{2n}}{\sqrt{n+1}}$ and $b_{k,n}= -\frac{\sqrt{2}}{\sqrt{n(n+1)}}$ if $k<n$.
\end{ejem}

\begin{remark} We want to point out that in Example $3$ the absolutely part of the measure $\nu$ is $\nu_a=\frac{1}{2}{\bf m}$ and the limits
$
\lim_{n\to \infty} b_{n-k,n}(\nu), \lim_{n\to \infty} b_{n-k,n}(\nu_a)
$ exist for every $k\geq 0$. This fact is true for general Toeplitz moment matrices as we will show in next section. Moreover, in this case
$
\lim_{n\to \infty} b_{n-k,n}(\nu)= \lim_{n\to \infty} b_{n-k,n}(\nu_a);
$
 we do not know if this is true in general.
\end{remark}

 Next we give a sufficient condition to assure that $\mathbf{A}$ is a classical inverse matrix of an HPD matrix
$\mathbf{M}$, not necessarily a moment matrix, verifying
$\lim_{n\to \infty} \lambda_n >0$:

\begin{prop} \label{prop1}
Let $\mathbf{M}=(c_{i,j})_{i,j=0}^{\infty}$ be an HPD matrix with $\lim_{n \to \infty} \lambda_n >0$. If $
\sum_{k=0}^{\infty} \vert c_{i,k} \vert^2 < \infty$ for all $i\in \NN_0$ then
$
\mathbf{A}\mathbf{M}= \mathbf{M}\mathbf{A} =\mathbf{I}.
$
\end{prop}

\begin{proof} First of all since $\mathbf{M}$ is Hermitian and $\sum_{k=0}^{\infty} \vert c_{i,k} \vert^2 < \infty$ for all $i\in \NN_0$ it follows  that all the rows and columns of $\mathbf{M}$ belong to $\ell_2$. On the other hand, since $\lim_{n\to \infty} \lambda_n > 0$ by Lemma $2$ the matrix
$\mathbf{A}$ is the  representation of a bounded operator on $\ell_2$ and, consequently, the rows and columns of
$\mathbf{A}$ belong to $\ell_2$. Then both matrices $\mathbf{A}\mathbf{M} $ and $\mathbf{M}\mathbf{A}
$ are well defined. We first show that for each $j,k \in \NN_0$
$$
\mathbf{A}\mathbf{M}[j,k]= \delta_{j,k}.
$$

We introduce the following notation: let $\mathcal{B}$
be a bounded operator on $\ell_2$ we denote by $\tilde{\mathcal{B}}_n=\Pi_n\mathcal{B} \Pi_n$,
where $\Pi_n$ is defined as in Lemma $2$, which matrix representation is given by
$\left( \begin{array}{cc}
               B_n & 0 \\
                           0 & 0 \\
                         \end{array}
               \right)$. It is well known (see e.g \cite{Halmos}) that $\{\tilde{\mathcal{{B}}}_n\}_{n=0}^{\infty}$ is strongly convergent to $\mathcal{B}$ on $\ell_2$. Since $\{\tilde{\mathcal{\overline{B}}}_n\}_{n=0}^{\infty}$ and
                                                $\{\tilde{\mathcal{\overline{B}}_n^{*}}\}_{n=0}^{\infty}$ are strongly convergent to $\mathcal{\overline{B}},\mathcal{\overline{B}}^{*}$ respectively then $\{\tilde{\mathcal{\overline{B}}}_n \tilde{\mathcal{\overline{B}}_n^{*}}\}_{n=0}^{\infty}$ is strongly convergent to $\mathcal{\overline{B}
                                                } \mathcal{\overline{B}}^{*}$. Fix
                                               $k\in \NN_0$, since  $(c_{i,k})_{i=0}^{\infty}\in \ell_2$ it follows that for every $n \geq k$ we have that
                                               $$\mathbf{A}\mathbf{M}e_k^{t}= \lim_{n\to \infty}
\left(
                         \begin{array}{cc}
                           \overline{\mathbf{B}_n\mathbf{B}_{n}^{*}} & 0 \\
                           0 & 0 \\
                         \end{array}
                                                \right)\left(\begin{array}{c}
                                c_{0,k} \\
                                c_{1,k} \\
                                \vdots  \\
                                c_{n,k} \\
                                c_{n+1,k} \\
                                \vdots
                              \end{array}\right)= \lim_{n\to \infty}  \left(
                         \begin{array}{cc}
                           \overline{\mathbf{B}_n\mathbf{B}_{n}^{*}} & 0 \\
                           0 & 0 \\
                         \end{array}
                                                \right)\left(\begin{array}{c}
                                c_{0,k} \\
                                c_{1,k} \\
                                \vdots  \\
                                c_{n,k} \\
                               0 \\
                                \vdots
                              \end{array}\right)=e_k.
$$
 Therefore
$$
e_j\mathbf{A}\mathbf{M}e_k^{t
}= \delta_{j,k}.
$$
 On the other hand, since $\mathbf{M}$ and $\mathbf{A}$ are Hermitian matrices $\mathbf{M}\mathbf{A}=\mathbf{I}$.
\end{proof}

\begin{remark} For an Hermitian matrix  $\mathbf{M}=(c_{i,j})_{i,j=0}^{\infty}$
 the condition $\sum_{k=0}^{\infty} \vert c_{i,k} \vert^2 < \infty$ for all $i\in \NN_0$ is equivalent to that
the mapping $\mathbf{M}:c_{00} \to \ell_2$ given by the formal matrix multiplication is well
defined.
\end{remark}

 Unfortunately Proposition $1$ does not provide any information for Hankel matrices. Indeed, in the indeterminate case (i.e. when
$\lim_{n\to \infty} \lambda_n >0$ by \cite{Berg-Szwarc}) the Hankel matrix $\mathbf{H}$ does not define a bounded operator on $\ell_2$;
moreover, $\mathbf{H}$ does not define even a linear mapping  from $c_{00}$ as we show in the following lemma:

\begin{lem} Let $\mathbf{H}=(s_{i+j})_{i,j=0}^{\infty}$ be a Hankel moment matrix such that $\lim_{n\to \infty} \lambda_n >0$. Then
$\sum_{n=0}^{\infty} s_n^2 =\infty$; in particular,  $\mathbf{H}$ does not define a linear
mapping from $c_{00}$.
\end{lem}

\begin{proof} Assume $\sum_{n=0}^{\infty}
s_n^2 < \infty$ and let $\{e_n\}_{n=0}^{\infty}$ be the standard basis of $\ell_2$. For every $n\in \NN_0$,
$s_{2n}=e_{2n}\mathbf{H} e_{2n}^{*}
 \geq \lambda_n$ and therefore $\lim_{n\to \infty} \lambda_n=0$.
\end{proof}

 We now particularize Proposition $1$ in the case of Toeplitz matrices; this case will be widely studied in next section:

\begin{corollary} Let $\mathbf{T}=(c_{j-i})_{i,j=0}^{\infty}$ be a Toeplitz HPD matrix verifying
$\lim_{n\to \infty} \lambda_n >0$.
If $\sum_{n=0}^{\infty} \vert c_n \vert^2 < \infty$ then $\mathbf{A}\mathbf{T}= \mathbf{T}\mathbf{A}=\mathbf{I}$.
\end{corollary}

 We finish this section with some applications of the above results. In \cite{Hofmaier} the follo\-wing problem is studied: given a sequence of polynomials $\{P_n(z) \}_{n=0}^{\infty}$ with $P_0(z):=1$ and
 $\deg(P_n(z))=n$ does there exist a measure $\mu$ with support on $\CC$ such that:
 $$
 \int P_n(z)\overline{P_m}(z)
 d\mu= \delta_{n,m}?
$$

 Note that once we know the sequence $\{P_n(z) \}_{n=0}^{\infty}$ we have the description of the transition matrix $\mathbf{B}$, thus
with our approach the above problem can be reformulated
in the following terms: given an upper triangular matrix
$\mathbf{B}
=(b_{k,n})_{k,n=0}^{\infty}$ with $b_{n,n}>0$,  is $\mathbf{B}=\mathbf{B}(\mu)$ the transition matrix associated with a certain measure $\mu$ ? Moreover, in the case of positive answer one can
ask about the density of polynomials in the corresponding space $L^{2}(\mu)$; concerning this problem by \cite{EGT} and Lemma $1$
we have the following result:

\begin{corollary} Let $\{P_n(z)\}_{n=0}^{\infty}$ be a sequence of polynomials with $P_n(z)=\sum_{k=0}^{n} b_{k,n}z^k$ and $b_{n,n}>0$ for every
$n\in \NN_0$. If there is a compactly supported measure $\mu$
such that $\mathbf{B}(\mu)=(b_{k,n})_{k,n=0}^{\infty}$
and $P^{2}(\mu)=L^{2}(\mu)$ then  $\mathbf{B}(\mu)$ does not define a bounded operator on $\ell_2$.
\end{corollary}

 We see two interesting cases of sequences of polynomials:

 {\it Case I. The sequence of polynomials $P_0(z)=1$ and $P_n(z)=b_nz^n$ with $b_n>0$}: in this case
$\mathbf{B}=(b_{i}\delta_{i,j})_{i=0}^{\infty}, \mathbf{A}=(b_{i}^{2}\delta_{i,j})_{i^,j=0}^{\infty}$ and $\mathbf{M}=(c_{i}\delta_{i,j})_{i,j=0}^{\infty}$
with $c_i=b_i^{-2}$ for all $i\geq 0$, all of them diagonal matrices.
By \cite{EGT} we have the following result:

\begin{corollary} Let $P_0(z):=1$ and $P_n(z)=b_nz^n$ with $b_n>0$. Assume that there exists a compactly supported measure $\mu$ such that $\mathbf{B}(\mu)=(b_{i}\delta_{i,j})_{i=0}^{\infty}$. If $P^{2}(\mu)=L^{2}(\mu)$ then $\lim_{n\to \infty} \frac{b_{n+1}}{b_n} \geq 1$ and
 $\lim_{n\to \infty} b_n=\infty$.
\end{corollary}

\begin{proof} Since the the support of $\mu$ is bounded
we have that $\limsup_{n\to \infty} \frac{c_{n+1,n+1}}{c_{n,n}} < \infty$. Moreover, if $\mathbf{M}$ is a moment matrix the entries verifies the Cauchy-Schwartz condition $c_{n,n}^{2} \leq c_{n-1,n-1}c_{n+1,n+1}$ for all $n\in \NN$, in particular $\{\frac{b_n}{b_{n+1}}\}_{n=1}^{\infty}$ is a non decreasing sequence and  $\lim_{n\to \infty} \frac{b_n}{b_{n+1}} \leq \infty$.  By \cite{EGT} and Corollary $2$ it follows that $\mathbf{B}$ does not define a bounded operator and consequently
$\limsup b_n=\infty$; then
$\lim_{n\to \infty} \frac{b_{n+1}}{b_n} \geq 1$ and $\lim_{n\to \infty} b_n=\infty$.
\end{proof}

 {\it Case II: The sequence of polynomials $P_0(z):=1$ and $P_n(z)=b_nz^{n-1}(z-1)$ with $b_n>0$}. The matrices $\mathbf{B}$ and $\mathbf{A}$ are:
\[
\mathbf{B}
=\left (
\begin{array}{ccccc}
1 & -b_1& 0 & 0 & \ldots \\
0  & b_1 & -b_2 & 0 & \ldots \\
 0 & 0 & b_2 & -b_3 & \ldots \\
 0  & 0 & 0 & b_3  & \ldots \\
  \vdots & \vdots & \vdots & \vdots & \ddots
\end{array}
\right )
 \qquad
\mathbf{A}=
\left (
\begin{array}{ccccc}
1+b_1^2 & -b_1^2 & 0 & 0 & \ldots \\
-b_1^2  &  b_1^2+b_2^2 &-b_2^2 & 0 & \ldots \\
 0 & -b_2^2 & b^2_2+ b_3^2 & -b_3^2 & \ldots \\
 0  & 0 & 0 & -b_3^2  & \ldots \\
  \vdots & \vdots & \vdots & \vdots & \ddots
\end{array}
\right ).
\]
 It can be checked that $\mathbf{A}$ is the classical inverse of the Hermitian matrix $\mathbf{M}$ given by
\[\mathbf{M}=  \left (
\begin{array}{ccccc}
c_0 & c_0 & c_0 & c_0 & \ldots \\
c_0  & c_1 & c_1 & c_1 & \ldots \\
 c_0 & c_1 & c_2 & c_2 & \ldots \\
 c_0  & c_1 & c_2 & c_3  & \ldots \\
  \vdots & \vdots & \vdots & \vdots & \ddots
\end{array}
\right )\]
 with $c_n=\sum_{k=0}^{n} \frac{1}{b_k^2}$.
In \cite{Hofmaier} Example 4.5.3. the authors asked
if there exists a measure for $\{P_n(z)\}_{n=0}^{\infty}$ in the case that $b_n=\frac{1}{\sqrt{(n-1)!}}$ for $n\geq 1$. The answer is negative, since in this case the matrix
$\mathbf{M}=(c_{i,j})_{i,j=0}^{\infty} $ is

\[\mathbf{M}=  \left (
\begin{array}{ccccc}
1 & 1 & 1 & 1 & \ldots \\
1  & 2 & 2 & 2 & \ldots \\
 1 & 2 & 3 & 3 & \ldots \\
1  & 2 & 3 & 4  & \ldots \\
  \vdots & \vdots & \vdots & \vdots & \ddots
\end{array}
\right )\]
 which obviously does not verify the condition  $c_{11}^{2} \leq c_{00}c_{22}$; consequently,
$\mathbf{M}$ is not a moment matrix.

\section{Inversion of Toeplitz moment matrices }

 Following Barria and Halmos \cite{Barria-Halmos} a bounded operator $\mathcal{A}:\mathcal{H} \to \mathcal{H}$ in a Hilbert space
$\mathcal{H}$ is a {\it weakly asymptotic Toeplitz operator } if the sequence $\mathcal{S}
^{*n}\mathcal{A}\mathcal{S}^{n}$ is strongly convergent, where $\mathcal{S}$ is the forward shift and $\mathcal{S}^{*}$ its adjoint. The limit is clearly a Toeplitz operator. In an analogous way we introduce the following definition:

\begin{definition} An infinite  matrix $\mathbf{A}=(a_{i,j})_{i,j=0}^{\infty}$ is a {\it weakly asymptotic Toeplitz matrix} if for every
$k\in \ZZ$ there exists
$
\lim_{n\to \infty} a_{n,n+k} =\alpha_k.
$
In such a case we denote by $\Lim(\mathbf{A})=(\alpha_{i-j})_{i,j=0}^{\infty}$ which is clearly a Toeplitz matrix.
\end{definition}
  It is easy
to show that if
$\mathcal{A}$ is a weakly asymptotic Toeplitz operator on a Hilbert $\mathcal{H}$ space then the matrix representation with respect to any orthonormal basis in $\mathcal{H}$ is a weakly asymptotic Toeplitz matrix.

 The main result in this section is:

\begin{theorem} \label{teorem1}
Let $\mathbf{T}$ be an infinite HPD  Toeplitz matrix and let $\mathbf{B}=(b_{k,n})_{k,n=0}^{\infty}$ be the transition matrix associated with $\mathbf{T}$. Assume that the matrix $\mathbf{A}=\overline{\mathbf{B}}\mathbf{B}^t=(a_{i,j})_{i,j=0}^{\infty}$ exists. Then:
\begin{enumerate}
\item $\mathbf{B}$ is a weakly  asymptotic Toeplitz matrix, i.e. for every $k\in \NN_0$ there exists
$
\lim_{n\to \infty} b_{n-k,n} =\beta_k.
$
\item $a_{i,k}= \sum_{j=0}^{i} \beta_j\beta_{k-j}$ for $i,k \in \NN_0$ with $i\leq k$.
\medskip
\item If $\sum_{k=0}^{\infty} \beta_k^2< \infty$
then
the matrix $\mathbf{A}$ is weakly asymptpotic Toeplitz and
$$
\lim_{k\to \infty} a_{k,k+j}= \sum_{i=0}^{\infty}\beta_i\beta_{i+j},
$$
 where the series above is absolutely convergent.
\end{enumerate}
\end{theorem}

\begin{proof}

 We first show that there exists $\lim_{n\to \infty} b_{n,n}$. Indeed, since
 $\mathbf{A}$ exists and $\mathbf{T}_n^{-1}=\overline{\mathbf{B}}_n\mathbf{B}_n^{t}$
$$
a_{0,0}=\sum_{k=0}^{\infty} \vert b_{0,k}\vert^2 =
 \lim_{n\to \infty} \sum_{k=0}^{n} \vert b_{0,k}\vert^2 =\lim_{n\to \infty} \mathbf{T}_n^{-1}[0,0].
 $$
 Using that the matrix $T_n^{-1}$ is persymmetric it follows that $\mathbf{T}_{n}^{-1}[0,0]= \mathbf{T}_n^{-1}[n,n]$ and therefore
$$
a_{0,0} = \lim_{n\to \infty}\mathbf{ T}_n^{-1}[n,n]= \lim_{n\to \infty} b_{n,n}^2.
$$
 Moreover $\beta_0=\lim_{n\to \infty} b_{n,n} =  \sqrt{a_{0,0}}>0
$ since $b_{n,n} >0$ for each $n\in \NN_0$.
\medskip
 Now we prove the existence of  $\lim_{n\to \infty} b_{n-k,n}$ for every $k\in \NN$. Indeed, using again that $\mathbf{T}_n^{-1}$ is persymmetric
$$
a_{0,k}= \lim_{n\to \infty} \mathbf{T}_n^{-1}[0,k]= \lim_{n\to \infty} \mathbf{T}_n^{-1}[n-k,n] = \lim_{n\to \infty} \overline{b}_{n-k,n}b_{n,n}.
$$
 Since $\lim_{n\to \infty} b_{n,n} =\beta_0 >0$ we have that there exists
$
\lim_{n\to \infty} b_{n-k,n}
$ and moreover,
$$
\lim_{n\to \infty} b_{n-k,n} = \dfrac{\overline{a}_{0,k}}{\beta_0}.
$$
 In a general way we can determinate all of the entries of the matrix $\mathbf{A}$. Indeed,  let $i,k\in \NN_0$ be fixed with $i\leq k$ then
$$
\mathbf{T}_{n}^{-1}[i,k]= \mathbf{T}_{n}^{-1}[n-i,n-k]=
(\overline{\mathbf{B}_n}\mathbf{B}_n^{t})[i,k]=
\sum_{j=0}^{i} \overline{b}_{n-i,n-i+j}b_{n-k,n-i+j}.
$$
 By passing to the limit and using Lemma $2$
$$
a_{i,k}
= \lim_{n\to \infty}\mathbf{ T}_{n}^{-1}[i,k]=
\sum_{j=0}^{i} \overline{\beta}_j\beta_{k+j-i} = \left(
                                              \begin{array}{llll}
                                               \overline{\beta}_0 & \overline{\beta}_1 & \cdots & \overline{\beta}_i \\
                                              \end{array}
                                            \right)\left(
                                                     \begin{array}{c}
                                                       \beta_{k-i} \\
                                                       \vdots \\
                                                       \beta_{k-1} \\
                                                       \beta_k \\
                                                     \end{array}
                                                   \right).
$$
 Thus, we have the description of $\mathbf{A}$. Note that the entries of the main diagonal of $\mathbf{A}$ are given by
$
a_{k,k}=
\sum_{i=0}^{k} \beta_i^{2}.
$

 In order to prove (3) let $j\in \NN_0$ be fixed and $k\in \NN$, then
$$
a_{k,k+j}= \left(
                                              \begin{array}{llll}
                                                \overline{\beta}_0 & \overline{\beta}_1 & \cdots & \overline{\beta}_k \\
                                              \end{array}
                                            \right)\left(
                                                     \begin{array}{c}
\mathcal{}                                                       \beta_{k+j-k} \\
                                                       \vdots \\
                                                       \beta_{k-1} \\
                                                       \beta_{k+j} \\
                                                     \end{array}
                                                   \right) = \sum_{i=0}^{k}\overline{\beta_i}\beta_{i+j}.
$$
 The series $\sum_{i=0}^{\infty}\beta_i\beta_{i+j}$ is absolutely convergent since $\sum_{k=0}^{\infty} \beta_k^2 < \infty$, and therefore there exists
$$
\lim_{k\to \infty} a_{k,k+j} = \sum_{i=0}^{\infty}\overline{\beta}_i\beta_{i+j}.
$$
\end{proof}

\begin{remark} Note that under the assumptions of Theorem $1$ for the matrix product
$\mathbf{A}=\overline{\mathbf{B}}\mathbf{B}^{t}$ it holds that
$$
\Lim(\mathbf{A}) =  \Lim(\overline{\mathbf{B}})\Lim(\mathbf{B}^t).
$$
\end{remark}
\medskip
 It is well known (see e.g. \cite{aki}) that for $\mathbf{T}$ being an infinite HPD Toeplitz matrix there exists a measure
$\nu$ with support on $\mathbb{T}$ such that $\mathbf{T}=\mathbf{T}(\nu)$. With this approach we
have the following consequences of Proposition above:

\begin{corollary} Let $\nu$ be a measure with support on $
\mathbb{T}$ and let
$P_n(z)=\sum_{k=0}^{n} b_{n-k,n}(\nu)z^k$ be the orthonormal polynomials associated with $\nu$. Assume that
$\lim_{n\to \infty} \lambda_n(\nu)>0$,
then for every $k\in \NN_0$ there exist
$
\lim_{n\to \infty} b_{n-k,n}(\nu).
$
\end{corollary}
 Using the results in \cite{Grenander-Szego} and Theorem \ref{teorem1} we have the following result:

\begin{corollary} Let $\nu$ be a measure with support on $\mathbb{T}$ which is absolutely continuous with respect to
${\bf m}$ and $d\nu(e^{i\theta})= \frac{w(\theta)}{2\pi}d\theta$ with $w(\theta) \in L^{1}[0,2\pi]$. If
$\essinf(w(\theta)) >0$ and $w(\theta) \in L^{2}[0,2\pi]$ then $\mathbf{A}(\nu)$ exists and
$$
\mathbf{T}(\nu)\mathbf{ A}(\nu) = \mathbf{ A}(\nu)\mathbf{T}(\nu)= \mathbf{I}.
$$
\end{corollary}

\begin{proof} By \cite{Grenander-Szego} it follows that the condition $\essinf (w(\theta))>0$ is equivalent to the fact
$\lim_{n\to \infty} \lambda_n(\nu)>0$. Consequently by Lemma \ref{lem2}
there exists $\mathbf{A}(\nu)$. Moreover,
for every $k\in \NN_0$
$$
c_{k} = \int_{\mathbb{T}} z^k d\nu(z)= \dfrac{1}{2\pi} \int_{0}^{2\pi} e^{ik}w(\theta)d\theta.
$$
 Since $w(\theta)\in L^{2}[0,2\pi]$  the Fourier coefficients of $w(\theta)$ belong to $\ell_2$ and
consequently the moment matrix $\mathbf{T}(\nu)$ verifies the assumptions in Proposition \ref{prop1} and it follows
that $\mathbf{A}(\nu)$ is the classical inverse of $\mathbf{T}(\nu)$ as we required.
\end{proof}

 Recall that the Hardy-Hilbert $\mathbf{H}^{2}$ space is the completion of the space $\mathbb{P}[z]$ in $L^{2}({\bf m})$; in this space the sequence
$\{z^n\}_{n=0}^{\infty}$ is an orthonormal basis.
 Let  $\varphi \in L^{\infty}(\mathbb{T})$ and consider the bounded operator
$\mathcal{M}_{\varphi} :L^{\infty}(\mathbb{T}) \to L^{\infty}(\mathbb{T})$ given by $\mathcal{M}_{\varphi}(g)=\varphi g$. Let $\mathcal{P}$ be the orthogonal proyection from $L^{\infty}(\mathbb{T})$ to $\mathbf{H}^{2}$, then $\mathcal{T}_{\varphi}:=\mathcal{P}\mathcal{M}_{\varphi}:\mathbf{H}^{2} \to \mathbf{H}^{2}$ is a bounded
Toeplitz operator which matrix representation is a Toeplitz matrix $\mathbf{T}_{\varphi}$. On the other hand, it is well known ( see e.g. \cite{Avendano}) that if
$\mathbf{T}=(t_{m-n})_{m,n=0}^{\infty}$ defines a bounded Toeplitz operator
$\mathcal{T}: \mathbf{H}^{2} \to \mathbf{H}^{2}$ there exists a function $\varphi \in L^{\infty}(\mathbb{T})$  such that for each $n\in \ZZ$:
$$
t_n =\dfrac{1}{2\pi} \int_{0}^{2\pi} \varphi(e^{i\theta})e^{-in\theta}d\theta .
$$
 In this case, $\mathcal{T}=\mathcal{T}_{\varphi}:=\mathcal{P}\mathcal{M}_{\varphi}$ and the matrix representation with respect the basis $\mathfrak{B}=\{z^n\}_{n=0}^{\infty}$
in $\mathbf{H}^{2}$ is the matrix $\mathbf{T}_{\varphi}=(t_{m-n})_{m,n=0}^{\infty}$. With this approach
Theorem $1$ can be reformulated in the following way:

\begin{prop} \label{prop2}
 Let $\mathbf{T}_{\varphi}=(c_{i-j})_{i,j=0}^{\infty}$ be a Toeplitz HPD matrix associated with  $\varphi \in L^{2}(\mathbb{T})$ and $\essinf   \varphi(z) >0$. Then the matrix
   $\mathbf{A}_{\varphi} =\overline{\mathbf{B}}\mathbf{B}^{t}$ exists and
   is the classical inverse matrix of $\mathbf{T}_{\varphi}$, i.e. $\mathbf{A}_{\varphi}\mathbf{T}_{\varphi}=\mathbf{T}_{\varphi}\mathbf{A}_{\varphi}=\mathbf{I}$.
\end{prop}

 In the following example we apply proposition  above to obtain the inverse matrix of a family of Toeplitz matrices in terms of the transition matrices:

\begin{ejem} Let $0<a<1$ be fixed and let $\mathbf{T}$ be
$$
\mathbf{T}= \dfrac{1}{1-a^2} \left(
     \begin{array}{ccccc}
       1+a^2 & a & 0 & 0 & \hdots \\
       a  & 1+a^2 & a & 0 & \hdots \\
       0 & a & 1+a^2 &  a & \hdots \\
       0 & 0 & a & a^2+1 & \hdots \\
       \vdots & \vdots  & \vdots  & \vdots
        & \ddots \\
     \end{array}
   \right).
$$
 $\mathbf{T}$ is an HPD Toeplitz matrix since it can be easily proved that
$\vert T_n \vert = \frac{1-a^{2(n+1)}}{(1-a^2)^{n}(1-a^2)}>0$ for every $n\in \NN_0$. Moreover,  $\mathbf{T}=\mathbf{T}_{\varphi}$ with continuous
symbol $\varphi(e^{i\theta})= \frac{1+a^2 +2a \cos(\theta)}{1-a^2}$
 verifying
$\inf_{\theta \in [0,2\pi]} \varphi(e^{i\theta}) =\frac{1-a}{1+a}>0$, therefore $\lim_{n\to \infty} \lambda_n >0$.

 By computing the Cholesky factorization of the matrix $\mathbf{T}$ and by an induction argument we can determinate the transition matrix $\mathbf{B}$ where if $k \leq n$ we have:
$$
b_{n-k,n} = \sqrt{1-a^2}(-1)^k a^{k}\frac{1-a^{2(n-k+1)}}{\sqrt{(1-a^{2(n+1)})(1-a^{2(n+2)})}} .
$$
 By passing to the limit we have that for every $k\in \NN_0$,
$$
\lim_{n\to \infty} b_{n-k,n}= \beta_k = (-1)^ka^{k}\sqrt{1-a^2}.
$$
 Once we have the coefficients $\beta_k's$ we may construct the infinite matrix $\mathbf{A}_{\varphi} $ and the Toeplitz matrix
$Lim(\mathbf{A})$. In particular, if $Lim(\mathbf{A})= (\alpha_{i-j})_{i,j=0}^{\infty}$ we have:
$$
\alpha_0= \lim_{k\to \infty} \mathbf{A}[k,k]= \lim_{k\to \infty} \sum_{j=0}^{k} \beta_j^2 = \sum_{k=0}^{\infty} (-1)^{2k}(1-a^2)a^{2k}=
(1-a^2)\sum_{k=0}^{\infty} a^{2k}=1
$$
$$
\alpha_n = \lim_{k\to \infty} \mathbf{A}[k,k+n] = \sum_{k=0}^{\infty} \overline{\beta}_k\beta_{k+n} =
\sum_{k=0}^{\infty} (-1)^k a^{k} \sqrt{1-a^2} (-1)^{k+n} a^{k+n} \sqrt{1-a^2}=
$$
$$
 (-1)^na^{n}(1-a^2)\sum_{k=0}^{\infty} a^{2k}= (-1)^na^{n}.
$$
 Therefore the inverse matrix $\mathbf{A}_{\varphi}$ of $\mathbf{T}_{\varphi}$  is
 \[ \mathbf{A}
 _{\varphi}  = \left (
 \begin{array}{ccccc}
 1-a^2 &  -a(1-a^2) & a^2(1-a^2) & -a^{3}(1-a^2) & \ldots \\
 -a(1-a^2) & 1-a^4 & -a(1-a^4) & a^2(1-a^4)  & \ldots \\
 a^2(1-a^2) & -a(1-a^4) & 1-a^6 &- a(1-a^6)& \ldots \\
 -a^{3}(1-a^2) & a^2(1-a^4)& -a(1-a^6) & 1-a^8 & \ldots \\
  \vdots & \vdots & \vdots & \vdots & \ddots
 \end{array}
 \right )\]
 and
$$
\Lim(\mathbf{A}_{\varphi}) = \left(
                   \begin{array}{ccccccc}
                     1 & -a & a^2 & -a^{3} & a^4 & -a^{5} & \hdots
                                                                              \\
                     -a & 1 & -a & a^2 & -a^{3} & a^4 & \hdots \\
                     a^2 & -a & 1 & -a & a^2 & -a^{3} & \hdots \\
                     -a^{3} & a^2 & -a & 1 & -a & a^2 & \hdots \\
                     a^4 & -a^{3} & a^2 & -a & 1 & -a & \hdots \\
                     -a^{5} & -a^{3} &  & a^2 & -a & 1 & \hdots \\
                     \vdots & \vdots & \vdots & \vdots & \vdots & \vdots  & \ddots \\
                   \end{array}
                 \right).
$$
 In this case it can be checked that $Lim(\mathbf{A}_{\varphi})= \mathbf{T}_{\frac{1}{\varphi}}$. This is always
true whenever the symbol $\varphi$ is continuous on $\mathbb{T}$ as we show in the following proposition.
\end{ejem}

\begin{prop} \label{prop3}
Let $\mathbf{T}_{\varphi}$ be an HPD Toeplitz matrix with continuous symbol $\varphi$ such that $\inf \varphi(z) >0$ then
$\mathbf{A}_{\varphi}$
is  weakly asymptotic Toeplitz and $
\Lim(\mathbf{A}_{\varphi}) = \mathbf{T}_{\frac{1}{\varphi}}.
$
\end{prop}

\begin{proof} Since $\frac{1}{\varphi}$ is continuous on $\mathbb{T}$ then $\mathbf{T}_{\frac{1}{\varphi}}$ defines a bounded Toeplitz
operator. Moreover, since $\varphi, \frac{1}{\varphi} \in L^{\infty}(\mathbb{T})$ by
\cite{Barria-Halmos} it follows that
$$
\mathbf{I}- \mathbf{T}_{\varphi}T_{\frac{1}{\varphi}}= \mathbf{H}_{z\varphi} \mathbf{H}_{z\frac{1}{\varphi}}=\mathbf{K}
$$
 Now, by \cite{Avendano} since the symbols $z\varphi$ and $z\frac{1}{\varphi}$ are continuous then the Hankel matrices
$\mathbf{H}_{z\varphi}, \mathbf{H}_{z\frac{1}{\varphi}}$ define compact operators and consequently
$\mathbf{K}$ defines a compact operator. Therefore $\mathbf{K}$ is a weakly asymptotic matrix with $\Lim(\mathbf{K})=0$. Since $\mathbf{A}$ defines a bounded operator which is the inverse operator of
$T_{\varphi}$ then
$$
\mathbf{T}
_{\varphi}(\mathbf{A}-\mathbf{T}_{\frac{1}{\varphi}}) = \mathbf{T}_{\varphi}\mathbf{A}- \mathbf{T}_{\varphi}\mathbf{T}_{\frac{1}{\varphi}}=\mathbf{K},
$$
 and
$$
\mathbf{A}-\mathbf{T}_{\frac{1}{\varphi}}= \mathbf{A}\mathbf{T}_{\varphi}(\mathbf{A}-\mathbf{T}_{\frac{1}{\varphi}})=\mathbf{A}\mathbf{K}.
$$
 Since $\mathbf{A}$ defines a bounded operator and $\mathbf{K}$ defines a compact operator then the matrix $\mathbf{A}\mathbf{K}$ is
the matrix representation of a compact operator and therefore such matrix is weakly asymptotically Toeplitz with limit $0$ and we have the conclusion.
\end{proof}

\begin{remark} We do not know if the above result is true if we consider symbols $\varphi$ essentially bounded, i.e.
such that $\varphi \in L^{\infty}(\mathbb{T})$ and such that ${\essinf} \varphi(z) >0$.
\end{remark}

\section{Smallest eigenvalues of the absolutely continuous part}

 Let $\nu$ be a measure with support on $\mathbb{T}$, by the Lebesgue decomposition $\nu=\nu_a + \nu_s$ where
$\nu_a$ is absolutely continuous with respect to ${\bf m}$ and $\nu_s$ is  singular with respect to ${\bf m}$. Let denote $P^{2}(\nu)$ the closure of $\mathbb{P}[z]$ in the space $L^{2}(\nu)$. In order to prove the main result of this section we need some lemmas.

\begin{lem} \label{lem4}
 Let $\mathbf{T}(\nu)$ be an HPD matrix associated with a measure $\nu$ with support on $\mathbb{T}$. Then the following are equivalent:
 \begin{enumerate}
\item $\lim_{n\to \infty} \lambda_n(\nu) =\lambda >0$.
\item The identity operator ${\it i}_{\nu}: (\mathbb{P}[z],P^{2}(\nu))
\to (\mathbb{P}[z],{\bf H}^{2})$ is bounded with norm $
\Vert {\it i}_{\nu}  \Vert = \lambda^{-1}.$
\end{enumerate}
\end{lem}

\begin{proof} Consider $(i_{\nu})_{n}:(\mathbb{P}_n[z],P^{2}(\nu)) \to (\mathbb{P}_n[z],{\bf H}^{2})$ the identity mapping, then
\begin{eqnarray*}
\Vert (i_{\nu})_{n} \Vert^{2} & = &\sup \{ \int \vert p(z) \vert^2d{\bf m} : p(z)\in \mathbb{P}_n[z], \int \vert p(z)\vert^2d\nu=1\}\\
\mbox{} & = &
\sup\{  vv^{*} :  \; v\in \CC^{n+1},
v\mathbf{T}_n(\nu)v^{*} =1 \}=\dfrac{1}{\lambda_n(\nu)}.
\end{eqnarray*}
 The result follows by taking the limit when $n$ tends to infinity.
\end{proof}

 The following result requires the same techniques used in \cite{Jarchow}:

\begin{lem} \label{lem5}
 Let $\mathbf{T}(\nu)$ be an HPD Toeplitz matrix associated with a measure $\nu$ with support on $\mathbb{T}$. Assume that ${\bf m}$ is absolutely continuous with respect to $\nu$, then the following are
equivalent:
\begin{enumerate}
\item $\lim_{n\to \infty} \lambda_n(\nu) =\lambda >0$.
\item The identity operator $f \to f$ denoted by ${\it I}_{\nu}: P^{2}(\nu)
\to {\bf H}^{2}$ exists and is bounded with norm $
\Vert {\it I}_{\nu}  \Vert = \lambda^{-1}.$
\end{enumerate}
\end{lem}

\begin{proof} $(2)$ implies $(1)$ is obvious. To prove $(1)$ implies $(2)$ by the above lemma
there exists $\lambda=\lambda(\nu)>0$ such that for every polynomial $p(z)\in \mathbb{P}[z]$
$$
\int \vert p(z)\vert^2 d {\bf m} \leq \lambda \int \vert p(z) \vert^2 d\nu.
$$
 Let now $f(z)\in P^{2}(\nu)$ and let $\{q_n(z)\}_{n=1}^{\infty}$ be a sequence
of polynomials which converges to $f(z)$ in the space  $P^{2}(\nu)$. For every $n,m\in \NN$,
$$
\int \vert q_n(z) -q_m(z) \vert^2 d {\bf m} \leq \lambda \int \vert q_n(z)-q_m(z) \vert^2 d\nu.
$$
 Then, the sequence $\{q_n(z)\}$ is a Cauchy sequence in the space  $L^{2}(\bf{m})$ and there exists a function $g(z) \in P^{2}(\bf{m})
$ such that $q_n(z) \to g(z)$ a.e. We have to show that $f(z)=g(z)$, $\nu$-a.e.
Since $\{q_n(z)\}_{n=1}^{\infty}$ is a Cauchy sequence in $L^{2}(\nu)$ then there exists a subsequence that we denote in the same way
which is pointwise convergent $\nu$-a.e. to
$g(z)$; i.e, there exists a measurable set $E$ with $\nu(E)=0$ such that
$q_n(z) \to g(z)$ if $z \notin E$. On the other hand, there exists a subsequence
of $\{q_n(z)\}_{n=1}^{\infty}$ that we denote in the same way and a measurable set $A\subset \mathbb{T}$ with ${\bf m} (A)=0$ such that  $q_n(z) \to f(z)$ if
 $z \notin A$. Since $\bf{m}$ is absolutely continuous with respect to $\nu$ we have that $f(z)=g(z)$ $\nu$-a.e.
Consequently $f(z)\in {\bf H}^{2}$ and the identity mapping  $f \to f$ exists and is continuous with
$\Vert {\it I}_{\nu}  \Vert = \lambda^{-1}.$
\end{proof}

 In the following result we prove the main result in the particular case of a singular measure.

\begin{lem} Let $\nu_s$ be a singular measure with infinite support on $\mathbb{T}$. Then,
$$
\lim_{n\to \infty} \lambda_n(\nu_s)=0.
$$
\end{lem}

\begin{proof} It is clear that $\nu_s$ does not satisfy Szeg\"{o} condition and it is well known, see \cite{Grenander-Szego} and  \cite{Conway} that a measure $\nu$ with support on $\mathbb{T}$ satisfies Szeg\"{o}  condition if and only if $P^{2}(\mu) = L^{2}(\mu)$. As a consequence of the results in \cite{EGT} we
obtain that
$$
\lim_{n\to \infty} \lambda_n(\nu_s)=0.
$$
\end{proof}

 We prove the main theorem in this section:

\begin{theorem} Let $\mathbf{T}(\nu)$ be the moment matrix associated with a positive measure $\nu$ with a measure with support on $\mathbb{T}$. Let
$\nu=\nu_a + \nu_s$ be the Lebesgue decomposition of $\nu$. Then,
$$
\lim_{n\to \infty} \lambda_n(\nu) = \lim_{n\to \infty} \lambda_n(\nu_a).
$$
\end{theorem}

\begin{proof} We first consider the case $\lim_{n\to \infty} \lambda_n(\nu)=0$. For every $p(z) =\sum_{k=0}^{n} v_k\in \mathbb{P}[z]$ and
$v=(v_0,\dots, v_n,0,\dots)\in c_{00}$
$$
v\mathbf{T}(\nu_a)v^{*}= \int \vert p(z) \vert^2d\nu_a \leq \int \vert p(z) \vert^2d\nu= v\mathbf{T}(\nu)v^{*},
$$
 consequently  $
\mathbf{T}(\nu_a) \leq \mathbf{T}(\nu)$.
 In particular, $
\lambda_n(\nu_a) \leq \lambda_n(\nu),
$ for every $n\in \NN$ and $\lim_{n\to \infty} \lambda_n(\nu_a)=0$.

 Assume now that $\lambda(\nu)= \lim_{n \to \infty} \lambda_n(\nu) > 0$. In
this case, by [EGT] we have that $P^{2}(\nu) \neq L^{2}(\nu)$, and
therefore by \cite{Hofmaier}
$$
\nu = \dfrac{1}{\vert h(z) \vert^2} {\bf m} + \nu_s,
$$
 where $h(z)= \dfrac{1}{\kappa} \sum_{n=0}^{\infty} \overline{P_n(0)}P_n(z) \in P^{2}(\nu)$ and $\kappa= \sum_{n=0}^{\infty} \vert P_n(0) \vert^2$. In particular, the absolutely part $\nu_a$  of $\nu$ coincides with
$
\nu_a= \dfrac{1}{\vert h(z) \vert^2} {\bf m}.
$
By \cite{Grenander-Szego} we have that
$$
\lim_{n\to \infty} \lambda_n(\nu_a)= \essinf \dfrac{1}{\vert h(z) \vert^2} = \dfrac{1}{\essup \vert h(z) \vert^2}.
$$
 Since ${\bf m} = \vert h(z) \vert^2 \nu$ with $\vert h(z) \vert^2 \in P^{2}(\nu)$ we have that ${\bf m}$ is absolutely continuous with respect to $\nu$ and by lemma \ref{lem5}  it follows that the identity mapping $I_{\nu}^{2}:P^{2}(\nu) \to {\bf H}^{2}$ given by $f\to f$ exists and is continuous with $\Vert I_{\nu}^{2} \Vert^2 = \dfrac{1}{\lambda(\nu)}$. Let
 $p(z) \in \mathbb{P}[z]$ we have that $p(z)h(z)\in P^{2}(\nu)$ and
$$
\int \vert p(z) h(z) \vert^2 d\nu \geq \lambda(\nu) \int \vert p(z) h(z) \vert^2d{\bf m}.
$$
 Since $\vert h(z) \vert^2\nu_s=0$ we have:
$$
\int \vert p(z) \vert^2  \vert h(z) \vert^2 d\nu_a \geq \lambda(\nu) \int \vert p(z) \vert^2 \vert h(z) \vert^2d{\bf m},
$$
 and since $\vert h(z) \vert^2 \nu_a = {\bf m}$ we have:
$$
\int \vert p(z)  \vert^2 d{\bf m}  \geq \lambda(\nu) \int \vert p(z) h(z) \vert^2d{\bf m}.
$$
 Therefore if $\Vert p(z) \Vert_{{\bf H}^2}=1$ we have that
$$
\int \vert p(z) \vert^2   \vert h(z) \vert^2 d{\bf m}  \leq  \dfrac{1}{\lambda(\nu)}
$$
 and then $\dfrac{1}{\lambda(\nu_a)}= \essup  \vert h(z) \vert^2 \leq \dfrac{1}{\lambda(\nu)}
$ and consequently $\lambda(\nu_a) \geq \lambda(\nu)$. Combining this result with the fact that $\lambda(\nu_a) \leq \lambda(\nu)$ we obtain the result.
\end{proof}

 Combining this theorem with the results in \cite{EGT} we have

\begin{corollary} \label{coro6} Let $\mu$ be measure with infinite support on $\overline{\mathbb{D}}$ and
$\nu= \mu/ \mathbb{T}$ with $\nu= \nu_a + \nu_s$. Then,
$$
\lim_{n\to \infty} \lambda_n(\mu)=\lim_{n\to \infty} \lambda_{n}(\nu)=\lim_{n\to \infty} \lambda_{n}(\nu_a).
$$
\end{corollary}

\begin{remark} For a moment matrix $\mathbf{M}(\mu)=(c_{i,j})_{i,j=0}^{\infty}$ associated with
a measure $\mu$ the support of $\mu$ is a subset of $\overline{\mathbb{D}}$ if and only if $\sup_{i,j \geq 0} \vert c_{i,j} \vert < \infty$. Indeed, if the support of $\mu$ is a subset of $\overline{\mathbb{D}}$ it obviously
follows that $\vert c_{i,j} \vert \leq \mu(\overline{\mathbb{D}})$ for all $i,j \geq 0$. The other part is a consequence of
the results in \cite{Atzmon}.
\end{remark}

 As a consequence of corollary \ref{coro6} the study of the asymptotic behaviour of the smallest eigenvalues of $\mathbf{M}=\mathbf{M}(\mu)$ associated with a measure $\mu$ supported on $\overline{\mathbb{D}}$ is reduced to the study of the same problem for the associated Toeplitz moment matrix $\mathbf{T}(\nu_a)$, being $\mu/\mathbb{T}= \nu=\nu_a + \nu_s$. In the sequel we obtain a way to find the Toeplitz matrix $\mathbf{T}(\nu)$
associated with $\mathbf{M}(\mu)$.

\begin{lem} \label{lem7}
Let  $\mathbf{M}(\mu)=(c_{i,j})_{i,j=0}^{\infty}$ be a moment matrix with $\sup_{i,j \geq 0} \vert c_{i,j} \vert < \infty$ associated with a measure
$\mu$.
 Then the following are equivalent:
\begin{enumerate}
\item $\mu/\mathbb{T}=0$.
\item $\mathbf{M}(\mu)$ is weakly asymptotic Toeplitz and $\Lim(\mathbf{M}(\mu))=0$.
\item $\lim_{n\to \infty} c_{n,n}=0$.
\end{enumerate}
\end{lem}

\begin{proof} We show $(1)$ implies $(2)$. Assume that $\mu/\mathbb{T}=0$ and let $k\in \ZZ$ be fixed and $n \in \NN$, then
$$
c_{n,n+k}=
\int_{\mathbb{D}} z^n\overline{z}^{n+k}d\mu
= \int_{\mathbb{D}}
\vert z \vert^{2n} \overline{z}^kd\mu.
$$
 Since $\{\vert z \vert^{2n} \overline{z}^k\}_{n=0}^{\infty} \to 0$ converges pointwise on $\mathbb{D}$ when $n\to \infty$, as a consequence of Egoroff's theorem
we have:
$$
\lim_{n\to \infty} c_{n,n+k}= \lim_{n\to \infty}  \int_{\mathbb{D}}
\vert z \vert^{2n} \overline{z}^kd\mu=0,
$$
 and $\mathbf{M}(\mu)$ is an asymptotic Toeplitz matrix. Part $(2)$ implies $(3)$ is trivial.  Assume now that $\lim_{n\to \infty} c_{n,n} =0$, then
$$
0= \lim_{n\to \infty} \int_{\overline{\mathbb{D}}} \vert z \vert^{2n}d\mu =
\lim_{n\to \infty} \int_{\mathbb{D}} \vert z \vert^{2n}d\mu+ \int_{\mathbb{T}} d\mu
= \mu(\mathbb{T}),
$$
 therefore $\mu/\mathbb{T}=0$, this shows $(3)$ implies $(1)$.
\end{proof}

\begin{corollary} \label{coro7} Let $\mathbf{M}(\mu)$ be a moment matrix associated with a measure
$\mu$ with support on $\overline{\mathbb{D}}$ and $\nu=\mu/\mathbb{T}$. Then $\mathbf{M}(\mu)$ is a weakly asymptotic Toeplitz matrix with
$\Lim(\mathbf{M}(\mu))=\mathbf{T}(\nu)$.
\end{corollary}

\begin{proof} Let $\eta=\mu/\mathbb{D}$ and $\nu=\mu/\mathbb{T}$, then
$\mathbf{M}(\mu)=\mathbf{M}(\eta)+ \mathbf{M}(\nu)$. By lemma \ref{lem7} $\mathbf{M}(\eta)$ is a weakly asymptotic
Toeplitz matrix and $\Lim(\mathbf{M}(\eta))=0$; on the other hand $\mathbf{T}(\nu)$ is a Toeplitz matrix, consequently
 $\mathbf{M}(\mu)= \mathbf{M}(\eta)+ \mathbf{M}(\nu)$ is an asymptotic Toeplitz matrix and
$\Lim(\mathbf{M}(\mu))=\mathbf{T}(\nu)$.
\end{proof}

\begin{remark} As a consequence of the above result if $\mathbf{M}(\mu)=(c_{i,j})_{i,j=0}^{\infty}$ is a moment matrix associated with a measure $\mu$ with support on $\overline{\mathbb{D}}$ and $\nu=\mu/\mathbb{T}$ then $\mathbf{T}(\nu)=(t_{j-i})_{i,j=0}^{\infty}$ with
$t_k= \lim_{n\to \infty} c_{n,n+k}$ for each $k\in \ZZ$.
\end{remark}

 We obtain some consequences concerning compactness of Hermitian moment matrices:

\begin{corollary} \label{coro8}Let $\mathbf{M}(\mu)$ be an HPD moment matrix associated with a measure $\mu$. If
 $\mathbf{M}(\mu)$ defines a compact operator on $\ell_2$ then $\mu/\mathbb{T}=0$.
\end{corollary}
\begin{proof} The result is a consequence of the fact that every compact operator on a Hilbert space
is a a weakly asymptotic Toeplitz operator with limit $0$.
\end{proof}

\begin{remark} The converse of corollary \ref{coro8} is not true even for Hankel matrices. Indeed, consider the Hilbert matrix $\mathbf{H}$:
$$
\mathbf{H}=
\begin{pmatrix}
  1    & 1/2 & 1/3 & 1/4   & \hdots     \\
  1/2  & 1/3 & 1/4 & 1/5   & \hdots  \\
  1/3  & 1/4 & 1/5 & 1/6
     & \hdots  \\
\vdots & \vdots & \vdots & \vdots &\ddots
\end{pmatrix}.
$$
 This is the matrix associated with the real Lebesgue
measure $\tau$ in the interval $[0,1]$ and $\tau / \mathbb{T}=0$. The matrix
$\mathbf{H}$ defines a bounded operator $\mathcal{H}$ from
 $\ell_2$ (see \cite{rosenblum}). Nevertheless,
$\mathbf{H}$ does not define a compact operator. As an easy proof of this fact consider the
weakly normalized sequence in $\ell_2$ given by $ x_{n} =
\frac{1}{\sqrt{n}}\left( \sum_{i=1}^{n} e_i \right)$. Note that
$$
\Vert \mathcal{H}(x_n) \Vert= \frac{1}{n} \sum_{m=0}^{\infty} \left
(\sum_{i=1}^{n}\frac{1}{i+n}\right)^{2} \geq
\frac{1}{n}\sum_{m=0}^{\infty} \frac{n^2}{(m+n)^2} \geq n
\int_{n}^{\infty} \frac{1}{x^2}dx = 1.
$$
 Consequently $\mathcal{H}$  is
not a compact operator.
\end{remark}

 We finish this section relating these results which the obtained in the preceding sections:

\begin{prop} Let $\mu$ be a positive measure with support on $\overline{\mathbb{D}}$ and let $\nu=\nu_a + \nu_s$
be the Lebesgue decomposition of  $\nu=\mu/\mathbb{T}$. If $\lim_{n\to \infty} \lambda_n(\mu)>0$ then

\begin{enumerate}
\item $$
\Vert \mathcal{A}(\mu) \Vert= \Vert \mathcal{A}(\nu) \Vert = \Vert \mathcal{A}(\nu_a) \Vert.
$$
\item
$$
\mathbf{A}(\nu)[0,0]= \mathbf{A}(\nu_a)[0,0].
$$
\end{enumerate}

\end{prop}
\begin{proof} By Corollary \ref{coro6} we have that
$$
\lim_{n\to \infty} \lambda_n(\mu)= \lim_{n\to \infty} \lambda_n(\nu)=\lim_{n\to \infty} \lambda_n(\nu_a),
$$

 and consequently by Lemma \ref{lem2}  the first part holds.
To prove the second part, by using Szeg\"{o} theory and denoting by $\Phi_n(z; \nu)$ the monic orthogonal polynomials associated to
a measure $\nu$ we have:
$$
\lim_{n\to \infty} b_{n,n}^{2}(\nu)= \lim_{n\to \infty} \dfrac{1}{\Vert \Phi_n(z; \nu) \Vert^2}=
\lim_{n\to \infty} \dfrac{1}{\Vert \Phi_n(z; \nu_a) \Vert^2}=\lim_{n\to \infty} b_{n,n}^{2}(\nu_a)
$$
 and therefore using the same arguments of Theorem \ref{teorem1}
$$
\mathcal{A}(\nu)[0,0]= \lim_{n\to \infty} b_{n,n}^{2}(\nu)=\lim_{n\to \infty} b_{n,n}^{2}(\nu_a)=\mathcal{A}(\nu_a)[0,0].
$$
\end{proof}

 The above proposition and remark $3$ suggest to us the following problem:

\noindent
 {\bf Problem 2.} Let $\nu$ be a measure with support on $\mathbb{T}$
verifying $\lim_{n\to \infty} \lambda_n(\nu)>0$, is it true that $\mathbf{A}(\nu)= \mathbf{A}(\nu_a)$? We point out that
in Example $3$ it is true that  $\mathbf{A}(\nu)= \mathbf{A}(\nu_a) = 2\mathbf{I}$.

 We have the following partial result concerning this problem:

\begin{corollary} Let $\nu$ be a positive measure with support on $\mathbb{T}$ verifying
$\lim_{n\to \infty} \lambda_n(\mu)>0$ and $\nu= \nu_a+\nu_s$ the Lebesgue decomposition of $\nu$. Assume that for any  $k \in \NN_0$
$$
\Lim(\mathbf{B}(\nu))= \Lim(\mathbf{B}(\nu_a)),
$$
 then $\mathbf{A}(\nu)= \mathbf{A}(\nu_a)$.
\end{corollary}

\begin{proof}
 By theorem $2$ it follows that $\lim_{n\to \infty} \lambda_n(\nu_a)>0$ and consequently both  matrices $\mathbf{A}(\nu),\mathbf{A}(\nu_a)$ exist and can be described in terms of
$\beta_k(\nu), \beta_k(\nu_a)$. Since by the assumptions
$$ \lim_{n\to \infty} b_{k,n}(\nu)= \beta_{k}(\nu)=\beta_{k}(\nu_a)= \lim_{n\to \infty} b_{k,n}(\nu_a)
$$

 the conclusion follows.
\end{proof}

 Using corollary above Problem $2$ can be reformulated in the following way:

\noindent
 {\bf Problem $2^{*}$
} Let $\nu$ be a measure with support on $\mathbb{T}$
verifying $\lim_{n\to \infty} \lambda_n(\nu)>0$, is it true that for every $k\in \NN$
$$
\lim_{n\to \infty} b_{n-k,n}(\nu) = \lim_{n\to \infty} b_{n-k,n}(\nu_a) ?
$$
\begin{remark} Note that for $k=0$ the equality  $
\lim_{n \to \infty} b_{n,n}(\nu)= \lim_{n\to \infty} b_{n}(\nu_a)$ is a consequence of Szeg\"{o} theory.
\end{remark}

\today


\begin{thebibliography}{Man96}

\bibitem{aki} Akhiezer,  N. I. \emph{The Classical Moment Problem}, Oliver and Boyd Ltd., Edinburgh and London, 1965.

\bibitem{Atzmon} Atzmon, A. \emph{A Moment Problem for Positive Measures on the
Unit Disc.} Pacific J. Math. {\bf 59} (1975) pp. 317--325.



\bibitem{Barria-Halmos} Barri\'{a}, J. and Halmos,P.R. \emph{Asymptoric Toeplitz  operators} Trans. Amer. Math. Soc.
{\bf 273} (1982) 621--630.



\bibitem{Mirta} Beckermann, B. and Castro, M. \emph{On the
determinacy of complex Jacobi matrices.} Math. Scand. {\bf 95(2)}
(2004) pp. 285--298.

\bibitem{Berg1} Berg, C. \emph{The multidimensional moment problem
and semigroups, Moments in Mathematics.} Proc. Syympos. Appl.
Math. {\bf 37} (1987) 110--124.


\bibitem{Berg} Berg, C. Chen,Y. and Ismail, M.E.H. \emph{Small
eigenvalues of large Hankel matrices: The indeterminate case.}
Math. Scand. {\bf 91} (2002) pp. 67--81.

\bibitem{Duran} Berg, C. and Duran, A.J. \emph{Orthogonal
Polynomials and analytic functions associated to
 positive definite
matrices.} J. Math. Anal. Appl. {\bf 315} (2006) pp. 54-67

\bibitem{Berg-Maserick} C. Berg, P. Maserick, \emph{Exponentially bounded positive definite functions},
Illinois J. Math. {\bf 28}(1984) 162--179.

\bibitem{Berg-Szwarc} Berg, C. and Szwarc, R. \emph{The smallest eigenvalue of Hankel matrices} Constr. Approx. {\bf 34} (2011), 107-133.

\bibitem{Bottcher} Bottcher, A. and  Silberman, B. \emph{Introduction to large truncated Toeplitz matrices},
UTX Springer, 1998.


\bibitem{Conway} Conway, J.B. \emph{The Theory of Subnormal
Operator.} Mathematical Surveys and Monographs, {\bf 36}, Amer.
Math. Soc., Providence, Rhode Island, 1991.


\bibitem{Conwayfuncional} Conway,J.B. \emph{A Curse in Functional Analysis}
       Graduate Texts in Mathematics, {\bf 96}, Springer-Verlag, New York, 1985.

\bibitem{Crone} Crone,L. \emph{A characterization of matrix operators on $\ell_2$}
Math. Z. {\bf 123} pp. 315-317 (1991)



\bibitem{EGT} Escribano,C; Gonzalo, R. and Torrano, E. \emph{ Small Eigenvalues of Large Hermitian
moment matrices.} J. Math. Anal. Appl. (2010)



\bibitem{EST} Escribano, C; Sastre, M.A.  and Torrano,E. \emph{Moment matrix of self-similar measures} Electronic Transactions on Numerical Analysis (ETNA),{\bf 24} (2006), 79-87

\bibitem{Grenander-Szego} Grenander, U. and Szego,G. \emph{Toeplitz
forms and their applications.} Chelsea Publishing Company New
York, 1955.



\bibitem{Emilio} Guadalupe, R. and Torrano, E. \emph{ On the Moment Problem in the Bounded
Case.} J. Comp. Appl. Math. {\bf 49}(1993) pp- 263--269.


\bibitem{Halmos} Halmos, P.R. \emph{A Hilbert Space Problem Book}  Second Edition
 Springer-Verlag New York  (1980)


\bibitem{Hofmaier} Hofmaier, F. \emph{Orthogonal Polynomials: Interaction Between Orthogonality in $L^2$-spaces and Orthogonality in Reproducing Kernel Spaces} P.H. Dissertation  Technische Universit\"{a} M\"{u}nchen (2007)


\bibitem{HJ} Horn, R.A. and Johnson, C.A. \emph{Matrix analysis}
Cambridge University Press, Cambridge, (1985).


\bibitem{Jarchow} Jarchow, H. \emph{Special Operators on Classical
Banach Spaces of Analytic Fucntions} Extrct. Mathematicae. {\bf 19
(1)} pp. 21--3, 2001.


 \bibitem{Jorgensen} Jorgensen,P.; Kornelson, K. and Shuman,K. \emph{Iterated function systems, moments, and transformations of infinite matrices}, Memoirs of the
American Mathematical Society, {\bf 213 , n 1003} (2011)

\bibitem{Avendano}  Mart\'{\i}nez-Avenda\~{n}o, R. A. and
Rosenthal, P.   \emph{An Introduction to
Operators on the
Hardy-Hilbert Space}, GTM {\bf 237}, Springer, 2007.

\bibitem{Provost} Provost,S.B. and Ha, H. \emph{On the inversion of certain moment matrices} Linear Alg. Appl. {\bf 430} pp. 2650--2658 (2009)

\bibitem{rosenblum} Rosenblum, M. \emph{On the Hilbert
matrix I and II} Proc. Amer. Math. Soc. (9) (1958) pp. 137-140,
581-585.





\bibitem{Sivakumar} Saranya, C.R. and Sivakumar,K.C. \emph{Generalized inverses of an invertible infinite matrix over a finite field}
Linear Algebra App. {\bf 418} pp. 468-479 (2006)


\bibitem{Szafraniec} Szafraniec, F.H.  \emph{Boundedness of the shift operator
related to definite positive forms: an application to moment
problem.} Ark. Math. {\bf 19} (1981) 251--259.

\bibitem{Szego}  Szeg\"{o}, G.  \emph{Orthogonal Polynomials}, Amer. Math. Soc., Colloquium Publications,
{\bf 32} (1939).



\bibitem{Trench-I} Trench, W.F. \emph{An algorithm for the inversion of finite Toeplitz matrices} J. Soc. Indust. Appl. Math. {\bf 12}
pp. 515-522 (1964)

\bibitem{Trench-II} Trench, W.F. \emph{An algorithm for the inversion of finite Hankel matrices} J. Soc. Indust. Appl. Math. {\bf 13(4)}
pp. 1102-1107 (1965)




\bibitem{Van Assche} Van Assche,W. \emph{Analytic aspects of
orthogonal polynomials.} Katholieke Universiteit Leuven, 1993.

\bibitem{zohar} Zihar, S. \emph{Toeplitz Matrix Inversion: The Algorithm of W.F. Trench }
Journal of the ACM {\bf 16(4)} (1969)

\end{thebibliography}
\end{document}